\documentclass{llncs}
\usepackage{amsmath}
\usepackage{amssymb}
\usepackage{enumerate}
\usepackage{varioref}
\usepackage{charter,eulervm}

\newcommand{\es}{\varnothing}
\renewcommand{\TH}{\mathfrak{T}\!\mathfrak{H}}
\newcommand{\FPT}{F\mspace{-1.5mu}P\mspace{-1.5mu}T}

\title{{\sc On the Threshold-Width of Graphs}}
\author{
 Maw-Shang~Chang\inst{1}
\and
 Ling-Ju~Hung\inst{1}
\and
 Ton~Kloks%
\thanks{This author is supported by the
National Science Council of Taiwan, under grants
NSC~98--2218--E--194--004 and NSC~98--2811--E--194--006.} \and
 Sheng-Lung~Peng\inst{2}
}

\institute{
 Department of Computer Science and Information Engineering\\
 National Chung Cheng University\\
 Min-Hsiung, Chia-Yi 621, Taiwan\\
 {\tt (mschang,hunglc)@cs.ccu.edu.tw}
\and
 Department of Computer Science and Information Engineering\\
 National Dong Hwa University\\
 Shoufeng, Hualien 97401, Taiwan\\
 {\tt slpeng@mail.ndhu.edu.tw}
}

\begin{document}

\maketitle

\begin{abstract}
The $\mathcal{G}$-width of a class of graphs $\mathcal{G}$ is defined
as follows. A graph $G$ has $\mathcal{G}$-width $k$ if there are
$k$ independent sets $\mathbb{N}_1,\dots,\mathbb{N}_k$ in $G$ such that
$G$ can be embedded into a graph $H \in \mathcal{G}$ such that
for every edge $e$ in $H$ which is not an edge in $G$, there exists an
$i$ such that both endpoints of $e$ are in $\mathbb{N}_i$.
For the class $\TH$ of threshold graphs we show that
$\TH$-width is NP-complete and we present fixed-parameter
algorithms. We also show that for each $k$, graphs of $\TH$-width
at most $k$ are characterized by a finite collection of forbidden
induced subgraphs.
\end{abstract}

\section{Introduction}
%%%%%%%%%%%%%%%%%%%%%%

\begin{definition}
\label{def G-width}
Let $\mathcal{G}$ be a class of graphs which contains all
cliques. The {\em $\mathcal{G}$-width\/} of a graph $G$
is the minimum number $k$ of independent sets
$\mathbb{N}_1,\dots,\mathbb{N}_k$ in $G$ such that there exists
an embedding $H \in \mathcal{G}$ of $G$ such that for
every edge $e=(x,y)$ in $H$ which is not an edge of $G$ there
exists an $i$ with $x,y \in \mathbb{N}_i$.
\end{definition}

We restrict the $\mathcal{G}$-width parameter to classes of
graphs that contain all cliques to ensure that it is well-defined
for every (finite) graph.

In this paper we investigate the width-parameter for the
class $\TH$ of threshold graphs and henceforth we call it
the {\em threshold-width\/} or {\em $\TH$-width\/}.
If a graph $G$ has threshold-width $k$ then we call $G$
also a {\em $k$-probe threshold graph\/}. We refer to the
{\em partitioned case\/} of the problem when the collection
of independent sets $\mathbb{N}_i$, $i=1,\dots,k$, which are
not necessarily disjoint, is a part
of the input.
A collection of independent sets $\mathbb{N}_i$, $i=1,\dots,k$,
is a {\em witness\/} for a partitioned graph.
For historical reasons we call the set of vertices
$\mathbb{P}=V-\bigcup_{i=1}^k \mathbb{N}_i$ the set of {\em probes\/}
and the vertices of $\bigcup_{i=1}^k \mathbb{N}_i$ the set of
{\em nonprobes\/}.

Threshold graphs were introduced in~\cite{kn:chvatal}
using a concept called `threshold dimension.'
There is a lot of information about threshold graphs in the
book~\cite{kn:mahadev}, and
there are chapters on threshold graphs
in the book~\cite{kn:golumbic2} and in the
survey~\cite{kn:brandstadt}.

We may take the following characterization as a
definition~\cite{kn:chvatal,kn:brandstadt}.

\begin{definition}
\label{def threshold}
A graph $G$ is a {\em threshold graph\/} if $G$ and its
complement $\Bar{G}$ are trivially perfect. Equivalently,
$G$ is a threshold graphs if $G$ has no induced
$P_4$, $C_4$, nor $2K_2$.
\end{definition}

We end this section with some notational conventions. For two sets
$A$ and $B$ we write $A+B$ and $A-B$ instead of $A \cup B$ and $A
\setminus B$. We write $A \subseteq B$ if $A$ is a subset of $B$
with possible equality and we write $A \subset B$ if $A$ is a subset
of $B$ and $A \neq B$. For a set $A$ and an element $x$ we write
$A+x$ instead of $A+\{x\}$ and $A-x$ instead of $A - \{x\}$. In
those cases we will make it clear in the context that $x$ is an
element and not a set.

A graph $G$ is a pair $G=(V,E)$ where the elements of $V$
are called the vertices of $G$ and where $E$ is a set of
two-element subsets of $V$, called the edges. We denote
edges of a graph as $(x,y)$ and we call $x$ and $y$ the
endvertices of the edge.
For a vertex $x$ we write $N(x)$ for its set of neighbors
and for $W \subseteq V$ we write
$N(W)=\bigcup_{x \in W} N(x)-W$ for the neighbors of $W$.
We write $N[x]=N(x)+x$ for the closed neighborhood of $x$.
For a subset $W$ we write $N[W]=N(W)+W$. Usually we will use $n=|V|$
to denote the number of vertices of $G$ and we will use
$m=|E|$ to denote the number
of edges of $G$.

For a graph $G=(V,E)$ and a subset $S \subseteq V$ of vertices
we write $G[S]$ for the subgraph {\em induced\/} by $S$, that is
the graph with $S$ as its set of vertices and with those
edges of $E$ that have both endvertices in $S$. For a subset
$W \subseteq V$ we will write $G-W$ for the graph $G[V-W]$ and
for a vertex $x$ we will write $G-x$ rather than $G-\{x\}$.
We will usually denote graph classes by calligraphic capitals.

In the next section we show that the class of graphs with
$\TH$-width at most $k$ is characterized by a finite collection of
forbidden induced subgraphs.

\section{A finite characterization}
%%%%%%%%%%%%%%%%%%%%%%%%%%%%%%%%%%%

A graph is a threshold graph if and only if it has no induced
$P_4$, $C_4$, or $2K_2$. We show that for any $k$, the class of
graphs with $\TH$-width at most $k$ is characterized by a finite
collection of forbidden induced subgraphs.

\begin{lemma}
\label{treedec}
A graph $G$ is a threshold graph if and only if it has a
binary tree-decomposition $(T,f)$, where $f$ is a bijection
from the vertices of $G$ to the leaves of $T$. Every internal
node of $T$, including the root is labeled either as a join--
or a union-node. For every internal node the right subtree
consists of a single leaf. Two vertices are adjacent
in $G$ if their lowest common ancestor in $T$ is a join-node.
\end{lemma}
\begin{proof}
According to Theorem~\vref{char threshold}
a graph is a threshold graph if and only if every induced subgraph
has either a isolated vertex or a universal vertex.
This proves the lemma.
\qed\end{proof}

\begin{theorem}
For every $k$ the class of graphs with $\TH$-width at most
$k$ is characterized by a finite collection of forbidden
induced subgraphs.
\end{theorem}
\begin{proof}
To prove this theorem we use the technique
introduced by Pouzet~\cite{kn:pouzet}.

Obviously, the class of graphs with $\TH$-width at most $k$
is hereditary. Let $k$ be fixed. Assume that the class of
$\TH$-width at most $k$ has an infinite collection of
minimal forbidden induced subgraphs, say $G_1,G_2,\ldots$.
In each $G_i$ single out one vertex $r_i$ and let
$G_i^{\prime}=G_i-r_i$. Then $G_i^{\prime}$ has
$\TH$-width at most $k$, thus there are independent sets
$\mathbb{N}^{(i)}_1,\ldots,\mathbb{N}^{(i)}_k$ in $G_i^{\prime}$
such that $G_i^{\prime}$ can be embedded into a threshold graph $H_i$
by adding certain edges between vertices that are pairwise contained
in some $\mathbb{N}^{(i)}_{\ell}$. For each $i$ consider a
binary tree-decomposition $(T_i,f_i)$
for $H_i$ as stipulated in Lemma~\ref{treedec}.
Each leaf is labeled by a $0/1$-vector with $k$ entries.
The $j^{\mathrm{th}}$ entry of this vector is equal to 0 or 1
according to whether the vertex is contained in $\mathbb{N}_j^{(i)}$
or not. Thus two vertices are adjacent in $G_i^{\prime}$ if and only
if their lowest common ancestor is a join-node and their vectors
are disjoint.

We give each leaf an additional $0/1$-label that indicates whether
the vertex that is mapped to that leaf is adjacent to $r_i$ or not.

Kruskal's theorem~\cite{kn:kruskal} states that binary trees,
with points labeled by a well-quasi-ordering are well-quasi-ordered
with respect to their lowest common ancestor embedding.
When we apply Kruskal's theorem to the
labeled binary trees $T_i$ that represent the graphs $G_i^{\prime}$
we may conclude that there exist $i < j$
such that $G_i^{\prime}$ is an induced subgraph of $G_j^{\prime}$.
But then we must also have that $G_i$ is an induced subgraph of
$G_j$. This is a contradiction because
we assume that the graphs $G_i$ are
minimal forbidden induced subgraphs.
This proves the theorem.
\qed\end{proof}

\section{$\TH$-width is fixed-parameter tractable}
%%%%%%%%%%%%%%%%%%%%%%%%%%%%%%%%%%%%%%%%%%%%%%%%%%

In this section we show that for constant $k$,
$k$-probe threshold graphs can be
recognized in $O(n^3)$ time.

The following is one of the many characterizations of threshold
graphs.

\begin{theorem}[\cite{kn:chvatal2,kn:harary,kn:manca,kn:orlin}]
\label{char threshold}
A graph is a threshold graph if and only if
every induced subgraph has an isolated vertex or
a universal vertex.%
\footnote{A vertex $x$ of a graph $G$ is
{\em isolated\/} if its neighborhood is the empty set. A vertex
$x$ of a graph $G=(V,E)$ is {\em universal\/} if $N[x]=V$.}
\end{theorem}

The following theorem is a monadic second-order characterization.
Problems that can be formulated in monadic second-order
logic can be solved
on graphs that have bounded rankwidth~\cite{kn:courcelle3}.
We will show that $k$-probe
threshold graphs have bounded rankwidth shortly.

\begin{theorem}
\label{char}
A graph $G=(V,E)$ has threshold-width at most $k$ if and only if
there exist $k$ independent sets $\mathbb{N}_i$, $i=1,\dots,k$,
such that for every $W \subseteq V$, $G[W]$ has an isolated
vertex or a vertex $\omega$ such that for every $y \in W-\omega$
either $\omega$ is adjacent to $y$ or there exists
$i \in \{1,\dots,k\}$
with $\{\omega,y\} \subseteq \mathbb{N}_i$.
\end{theorem}
\begin{proof}
This is inferred by Theorem~\ref{char threshold}
and Definition~\ref{def G-width}.
\qed\end{proof}

\begin{definition}[\cite{kn:oum,kn:oum4}]
A {\em rank-decomposition\/} of a graph
$G=(V,E)$ is a pair $(T,\tau)$
where $T$ is a ternary tree and $\tau$ a bijection
from the leaves of $T$ to the vertices of $G$.
Let $e$ be an edge in $T$ and consider the two
sets $A$ and $B$ of leaves of the two subtrees of $T-e$.
Let $M_e$ be the submatrix of the adjacency matrix of
$G$ with rows indexed by the vertices of $A$ and columns
indexed by the vertices of $B$. The {\em width of $e$\/} is the
rank over $GF(2)$ of $M_e$. The {\em width of $(T,\tau)$\/}
is the maximum width over all edges $e$ in $T$ and
the {\em rankwidth\/} of $G$ is the minimum width
over all rank-decompositions of $G$.
\end{definition}

Computing the rankwidth of a graph is NP-complete~\cite{kn:hlineny2}
but it is fixed-parameter
tractable. This can be seen in various ways:
~\cite{kn:oum3} Proves that there is
a finite obstruction set for fixed-parameter rankwidth.
Now, note that
Schrijver describes a
general algorithm to minimize a class of
submodular functions which uses the
ellipsoid method~\cite[Chapter 45]{kn:schrijver}.
He turns this into a `combinatorial
algorithm' for a seemingly larger class of submodular
functions, in~\cite{kn:schrijver2}.
Using this result,~\cite{kn:hlineny2} describes a combinatorial
fixed-parameter algorithm
for computing the rankwidth of matroids.

\begin{lemma}
Threshold graphs have rankwidth at most one.
\end{lemma}
\begin{proof}
The class of graphs with rankwidth at most 1 is exactly the
class of distance-hereditary graphs~\cite{kn:oum4}. Every
threshold graph is distance hereditary
(see, {\em e.g.\/},~\cite{kn:brandstadt,kn:harary,kn:manca}).
\qed\end{proof}

\begin{theorem}
$k$-Probe threshold graphs have rankwidth at most $2^k$.
\end{theorem}
\begin{proof}
Consider a rank-decomposition $(T,\tau)$
with width 1 for an embedding $H$ of $G$.
Consider an edge $e$ in $T$ and assume that $M_e$ is
an all-1s-matrix. Each independent set $\mathbb{N}_i$
creates a 0-submatrix in $M_e$. If $k=1$ this proves that
the rankwidth of $G$ is at most 2. In general, for $k \geq 0$,
note that
there are at most $2^k$ different neighborhoods from
one leaf-set of $T-e$ to the other. It follows
that the rank of $M_e$
is at most $2^k$. By the way, it is easy to see that this matrix
has indeed rank $2^k$ in the worst case.
\qed\end{proof}

\begin{theorem}
\label{unfixed}
For each $k \geq 0$ there exists an $O(n^3)$ algorithm
which checks whether a graph $G$ with $n$
vertices is a $k$-probe threshold graph. Thus
$\TH$-width $\in \FPT$.
\end{theorem}
\begin{proof}
$k$-Probe threshold graphs have bounded rankwidth.
$C_2MS$-Problems can be solved in $O(n^3)$
time for graphs of bounded rankwidth%
~\cite{kn:courcelle3,kn:hlineny,kn:oum4}. By Theorem~\ref{char},
the recognition of $k$-probe threshold graphs is such a problem.

Alternatively, the theorem is also proved by using the finite
collection of forbidden induced subgraphs. Note however that this
proof is non-constructive; Kruskal's theorem does not provide
the forbidden induced subgraphs.
\qed\end{proof}

{\em A fortiori\/}, Theorem~\ref{unfixed} holds
as well when the collection of independent sets
$\mathbb{N}_1,\dots,\mathbb{N}_k$ is a part of the
input. Thus for each $k$ there is an $O(n^3)$ algorithm
that checks whether a graph $G$, given with $k$ independent sets
$\mathbb{N}_i$, can be embedded into a threshold graph.

There are a few drawbacks to this solution. First of all,
Theorem~\ref{unfixed} only shows the {\em existence\/}
of an $O(n^3)$ recognition algorithm;
{\em a priori\/}, it is unclear how
to obtain an algorithm explicitly.
Furthermore, the constants involved in the algorithm
make the solution impractical; already there is an
exponential blow-up when one moves from threshold-width to
rankwidth.

In the next section we show that there exists an explicit,
linear-time
algorithm for the recognition of partitioned
$k$-probe threshold graphs.

\section{Recognition of partitioned $k$-probe threshold graphs}
%%%%%%%%%%%%%%%%%%%%%%%%%%%%%%%%%%%%%%%%%%%%%%%%%%%%%%%%%%%%%%%%
\label{section fixed}

In this section, let $(G,\mathcal{N})$ be a partitioned $k$-probe
threshold graph, consisting of a graph $G$ and a $k$-witness
$\mathcal{N}$.

\begin{lemma}
\label{isolated}
If $G$ has an isolated vertex $x$ then $G$ is partitioned
$k$-probe
threshold if and only if $G-x$ is partitioned $k$-probe threshold
with the induced collection of independent sets. The same statement
holds as well for the unpartitioned case.
\end{lemma}
\begin{proof}
Assume $G$ is $k$-probe threshold. Consider an embedding $H$ of $G$.
Then $H-x$ is an embedding of $G-x$. Thus $G-x$ is $k$-probe
threshold.

\noindent
Assume $G-x$ is $k$-probe threshold. Let $H^{\prime}$ be an embedding
of $G-x$. Then we obtain an embedding of $G$ by adding $x$ as an isolated
vertex to $H^{\prime}$.
\qed\end{proof}

\begin{theorem}
\label{fixed}
For every $k$ there exists a linear-time algorithm to check
whether a pair $(G,\mathcal{N})$, where $G$ is a graph and
$\mathcal{N}$ a collection of $k$ independent sets in $G$,
is a partitioned $k$-probe threshold graph.
\end{theorem}
\begin{proof}
Assume that $(G,\mathcal{N})$ is a partitioned
graph
and let $H$ be an embedding
of $G$. If $H$ has an isolated vertex $x$, then $x$ is also
isolated in $G$ since $H$ is an embedding of $G$.
By Lemma~\ref{isolated} any isolated
vertex may be safely removed from $G$.

\noindent Now we may assume that any embedding $H$ is connected. By
Theorem~\ref{char threshold} $H$ has a universal vertex $\omega$. We
call $\omega$ a `probe universal vertex' of $(G,\mathcal{N})$ if for
every nonneighbor $z$ there is an independent set in $\mathcal{N}$
which contains both $\omega$ and $z$. Thus any partitioned $k$-probe
threshold graph has an isolated vertex or a probe universal vertex.
Finally, observe the following: if $\omega$ is a probe universal
vertex then $G$ is $k$-probe threshold if and only if $G-\omega$ is
$k$-probe threshold, since we may add $\omega$ as a universal vertex
to any embedding of $G-\omega$ and obtain an embedding of $G$. Since
$k$ is a constant, an elimination ordering by isolated - and probe
universal vertices can be obtained in linear time. \qed\end{proof}

\begin{remark}
Note that the algorithm described in Theorem~\ref{fixed}
is fully polynomial. This proves that the `sandwich problem,'
studied by Golumbic {\em et al.\/}, in~\cite{kn:golumbic},
is polynomial for threshold graphs.
\end{remark}

\section{$\TH$-width is NP-complete}
%%%%%%%%%%%%%%%%%%%%%%%%%%%%%%%%%%%%

Let $\mathfrak{T}$ be the class of complete
graphs (cliques). We proved in~\cite{kn:chang} that
$\mathfrak{T}$-width is
NP-complete. For completeness sake we include the proof.

\begin{theorem}
\label{T-width}
$\mathfrak{T}$-Width is NP-complete.
\end{theorem}
\begin{proof}
Let $(G,\mathcal{N})$
be a partitioned $k$-probe complete graph,
with a witness
\[\mathcal{N}=\{\mathbb{N}_i\;|\; i=1,\dots,k\}\]
which is a collection of $k$ independent sets in $G$.
Thus every non-edge of $G$ has both its endvertices
in one of the independent sets $\mathbb{N}_i$.
Then $\mathcal{N}$ forms a
clique-cover of the edges of $\Bar{G}$.
This shows that a graph $G$ has $\mathfrak{T}$-width
at most $k$ if and only if the edges of $\Bar{G}$
can be covered with $k$ cliques. The problem to
cover the edges of a graph by a minimum number of
cliques is NP-complete~\cite{kn:kou}. This proves the theorem.
\qed\end{proof}

\begin{theorem}
\label{TH-width}
$\TH$-width is NP-complete.
\end{theorem}
\begin{proof}
Assume there is a polynomial-time algorithm to compute
$\TH$-width. We show that we can use that algorithm to
compute $\mathfrak{T}$-width.
Let $G$ be a graph for which we wish to
compute $\mathfrak{T}$-width.
Construct a graph $G^{\prime}$ by adding a clique $C$
with $n^2$ vertices. Make all vertices of $C$ adjacent to
all vertices of $G$. Add one more vertex $\omega$ and make
$\omega$ also adjacent to all vertices of $G$.
Consider two nonadjacent vertices $x$ and $y$
of $G$. In any embedding of $G^{\prime}$ into a threshold graph,
either $x$ and $y$ are adjacent or $\omega$ is adjacent to all
vertices of $C$. However, to make $\omega$ adjacent to all
vertices of $C$, we need at least $n^2$ independent sets.
Obviously, making a clique of $G$ embeds $G^{\prime}$ into
a threshold graph, namely the complement of a star and a collection
of isolated vertices. This embedding needs less than
$n^2$ independent sets.
This proves the theorem.
\qed\end{proof}

\section{A fixed-parameter algorithm to compute $\TH$-width}
%%%%%%%%%%%%%%%%%%%%%%%%%%%%%%%%%%%%%%%%%%%%%%%%%%%%%%%%%%%%

Assume that $(G,\mathcal{N})$ is a connected
partitioned $k$-probe threshold graph
with witness
\[\mathcal{N}=\{\mathbb{N}_i \;|\; i=1\dots,k\}\]
and let $H$ be an embedding. The {\em label\/} $L(x)$ of a
vertex $x$ is the $0/1$-vector of length $k$ with the $i^{\mathrm{th}}$
entry $L_i(x)$ equal to 1 if and only if $x \in \mathbb{N}_i$.
We write $L(x) \leq L(y)$ if $L_i(x) \leq L_i(y)$ for all $i=1,\dots,k$.
We write $L(x) \perp L(y)$ if there is no $i$ with $L_i(x)=L_i(y)=1$.

\begin{definition}
A witness $\mathcal{N}$ is {\em well-linked\/} if
for every $i=1,\dots,k$, every vertex
$x \not\in \mathbb{N}_i$ has a neighbor in $\mathbb{N}_i$.
\end{definition}

\begin{lemma}
Every $k$-probe threshold graph has a witness with $k$
independent sets which is well-linked.
\end{lemma}
\begin{proof}
Starting with any witness, repeatedly add a vertex $x$ to an
independent set $\mathbb{N}_i$ if it has no neighbor in that set.
\qed\end{proof}

Consider the equivalence relation $\equiv$ defined by
$x \equiv y$ if $N(x)=N(y)$. Denote the equivalence class
of a vertex $x$ by $(x)$.
Define the partial order $\preceq$ by:
\[(x) \preceq (y) \quad\mbox{if}\quad N(x) \subseteq N(y).\]

Likewise, we consider the equivalence relation $\equiv^{\prime}$ defined
by $x \equiv^{\prime} y$ if $N[x]=N[y]$. The equivalence class of a vertex
$x$ under this relation is denoted by $[x]$.
We consider the
partial order defined by:
\[ [x] \preceq [y] \quad\mbox{if}\quad N[x] \subseteq N[y].\]

\begin{lemma}
Assume $(G,\mathcal{N})$ is a $k$-probe clique with a well-linked
witness $\mathcal{N}$. Then
\[(x) \preceq (y) \quad\Leftrightarrow\quad
L(x) \geq L(y) \neq \mathbf{0}.\]
\end{lemma}
\begin{proof}
Assume $(x) \preceq (y)$. Thus $N(x) \subseteq N(y)$.
Assume that
$y \in \mathbb{N}_i$ and $x \not\in \mathbb{N}_i$ for some $i$. Since
$\mathcal{N}$ is well-linked there exists a vertex
$z \in N(x) \cap \mathbb{N}_i$. Then $z \in N(y)$ since
$N(x) \subseteq N(y)$.
But this contradicts $\{y,z\} \subseteq \mathbb{N}_i$.

\noindent
Assume $L(x) \geq L(y) \neq \mathbf{0}$. Then $x$ and $y$
are not adjacent. Let $z \in N(x)$. Then $L(z) \perp L(x)$.
Since $L(x) \geq L(y)$ also $L(z) \perp L(y)$. Thus $z \in N(y)$
since $(G,\mathcal{N})$ is a $k$-probe clique.
\qed\end{proof}

Note that Definition~\ref{def threshold} is equivalent to the
following characterization.

\begin{theorem}[\cite{kn:chvatal2,kn:mahadev}]
\label{chain char}
A graph $H$ is a threshold graph if and only if for every
pair of vertices $x$ and $y$, $N(x) \subseteq N[y]$ or
$N(y) \subseteq N[x]$.
\end{theorem}

\noindent
In other words, a graph $G=(V,E)$ is a threshold
graph if and only if there is a {\em total order\/}
of the vertices $[x_1,\dots,x_n]$, {\em i.e.\/}, a {\em chain\/},
such that:
\[ 1 \leq i < j \leq n \quad\Rightarrow\quad N(x_i) \subseteq N[x_j].\]

\begin{theorem}
\label{max nonprobes}
Let $(G,\mathcal{N})$ be a $k$-probe threshold graph with a
well-linked witness $\mathcal{N}$ and let $H$ be an embedding.
For every nonadjacent pair $x$ and
$y$ in $G$ with $N_H(x) \subseteq N_H[y]$:
\[ (x) \preceq (y) \quad\Leftrightarrow\quad L(x) \geq L(y).\]
\end{theorem}
\begin{proof}
Assume $L(x) \geq L(y)$. Let $z \in N_G(x)$. Then $z \in N_H[y]$.
Since $x$ and $y$ are not adjacent, $z \neq y$. Thus $z \in N_H(y)$.
If $z \not\in N_G(y)$, then there exists an $i$ with
$\{z,y\} \subseteq \mathbb{N}_i$. Now $L(x) \geq L(y)$ implies that
also $x \in \mathbb{N}_i$, which contradicts that $z$ is adjacent to $x$.
Hence $(x) \preceq (y)$.

\noindent
Assume $(x) \preceq (y)$, that is, $N_G(x) \subseteq N_G(y)$.
{\em A fortiori\/}, $x$ and $y$ are not adjacent.
Assume $\neg(L(x) \geq L(y))$. Then there exists an $i$ with
$y \in \mathbb{N}_i$ and $x \not\in \mathbb{N}_i$. Since
$\mathcal{N}$ is well-linked, there exists a vertex
$z \in N_G(x) \cap \mathbb{N}_i$. Since $(x) \preceq (y)$,
$z \in N_G(y)$, contradicting that $z$ and $y$ are both in $\mathbb{N}_i$.
\qed\end{proof}

For completeness sake we note the following.

\begin{lemma}
Assume $[x] \preceq [y]$ and $x \neq y$. Then
\[ \forall_{i \in L(y)} N_G(x) \cap \mathbb{N}_i =\{y\}.\]
\end{lemma}
\begin{proof}
Since $x$ and $y$ are adjacent, we have that $L(x) \perp L(y)$.
Assume that $y \in \mathbb{N}_i$ for some $i \in \{1,\dots,k\}$.
Thus $x \not\in \mathbb{N}_i$. Since $\mathcal{N}$ is well-linked,
there exists a vertex $z \in N(x) \cap \mathbb{N}_i$.
Since $N_G[x] \subseteq N_G[y]$, $z \in N_G[y]$. But then we must have
$z=y$, otherwise $z$ and $y$ are nonadjacent.
\qed\end{proof}

\begin{definition}
A {\em true --\/} or {\em false module\/}
is a set of vertices such that every pair is a true -- or false
twin, respectively.\footnote{A {\em true twin\/} is a pair of vertices
$x$ and $y$ with $N[x]=N[y]$. A {\em false twin\/} is a pair of
vertices $x$ and $y$ with $N(x)=N(y)$.}
A {\em $k$-probe module\/} is either a false module with at least
3 vertices or a true module with at least $k+3$ vertices.
\end{definition}

\begin{lemma}
Let $S$ be $k$-probe module.
Then $G$ has $\TH$-width
at most $k$ if and only if $G-x$ has $\TH$-width at most $k$ for
any $x \in S$.
\end{lemma}
\begin{proof}
If $G$ is $k$-probe threshold then so is $G-x$ for any vertex $x$.
Let $x \in S$ and assume that $G-x$ is a $k$-probe threshold graph.
Let $H$ be an embedding of $G-x$.
First assume that $S$ is a false module with at least three vertices.
Let $y \in S-x$. If $y$ is in the independent set, then we can let
$x$ be a copy of $y$. Assume that all vertices of $S-x$ are in the
clique of $H$. Since $S-x$ has at least two vertices, they must be
nonprobes. We can let $x$ be a copy of either of them.

\noindent
Assume $S$ is a true module with at least $k+3$ vertices.
Then at least $k+1$ vertices are in the clique $C$ of $H$.
Let $z$ be a vertex of $S \cap C$ with a minimal
closed neighborhood in $H$.
Assume that $z$ has a neighbor $u$ in $H$
which is not a neighbor of $z$ in $G$. Then $u$ is a neighbor
of every vertex of $S \cap C$ in $H$, but not in $G$. Since every
pair of vertices $a,b \in S$ is adjacent in $G$, $L(a) \perp L(b)$.
It follows that $u$ must be in at least $k+1$ independent
sets, which is a contradiction. Thus $N_H(z)=N_G(z)$, and
we can let $x$ be a copy of $z$.
\qed\end{proof}

\begin{definition}
A vertex $x$ is {\em maximal\/} if there exists
no $(y) \neq (x)$ with $(x) \preceq (y)$ and there
exists no $[y] \neq [x]$ with $[x] \preceq [y]$.
\end{definition}

\begin{lemma}
\label{bound on max}
Assume that $G$ is a $k$-probe threshold graph
without $k$-probe module. Then there are
at most $2^{k+1}+k$ maximal vertices.
\end{lemma}
\begin{proof}
Consider a well-linked
embedding $H$. By Theorem~\ref{chain char} there is a
chain order of its vertices. Let $M_0,M_1,\dots$ be the
equivalence classes in $H$ of vertices with the same
open or closed neighborhoods.
Assume they are ordered such that $N[x_i] \supseteq N(x_{i+1})$
for each $x_i \in M_i$ and $x_{i+1} \in M_{i+1}$, for
$i=0,1,\dots$. Thus if $H$ is connected, $M_0$ is the set of
universal vertices in $H$. We call these equivalence classes
$M_0,M_1,\dots$ the
{\em levels\/} of the embedding. Thus a level contained
in the clique induces a $k$-probe clique in $G$
and a level contained
in the independent set induces an independent set in $G$.

\noindent
Consider the partition of each level $M_s$
into sets of vertices with
the same label. We call the sets of the partition of
a level $M_s$ the {\em label-sets\/} of $M_s$. Notice that
each label-set is a module in $G$. Since there is no
$k$-probe module, each label-set of nonprobes has at most
2 vertices and the label-set of probes has at most $k+2$
vertices. Thus
\[ |M_s| \leq 2(2^k-1)+(k+2)= 2^{k+1} + k. \]

By Theorem~\ref{max nonprobes} a vertex $x \in M_s$ is maximal
if it has a label $L(x)$ such that
all other label-sets $L^{\prime} \leq L(x)$ in
$M_0,\dots,M_s$ are empty. It follows that there are at most
$\sum_{i=0}^k \binom{k}{i}=2^k$ label-sets of maximal vertices,
at most $2^k-1$ of maximal nonprobes,
each containing at most 2 elements, and
at most one label-set
of maximal probes, containing at most $k+2$ elements.
Thus the number of maximal elements is bounded
by $2^{k+1}+k$.
\qed\end{proof}

\begin{lemma}
\label{super universal}
Assume $G$ is a $k$-probe threshold graph without isolated
vertices and without $k$-probe module.
There exists a set $\Upsilon$, of size $|\Upsilon| \leq 2^{2(k+1)}$
such that any well-linked embedding of $G$ has its set of universal
vertices
$M_0 \subseteq \Upsilon$. This set $\Upsilon$ can be computed in
linear time.
\end{lemma}
\begin{proof}
Since $G$ has no isolated vertices, $H$ has a set of
universal vertices $M_0$. Start with
$\Upsilon=\es$. Repeatedly compute the set of maximal
vertices in $G$, add them to $\Upsilon$,
and delete them from the graph. After at most
$2^k$ repetitions, each label-set of $M_0$ is contained
in $\Upsilon$.
Since each set of maximal elements has at most $2^{k+1}+k$
vertices,
\[ |\Upsilon| \leq 2^k(2^{k+1}+k) \leq 2^{2k+1}+ 2^{2k} \leq 2^{2(k+1)}.\]
\qed\end{proof}

\begin{definition}
A {\em probe universal set\/} is a set $U$ of labeled vertices
such that for every vertex $x \not\in U$, there is a label
for $x$ such that $U+x$ is a partitioned probe clique.
\end{definition}

\begin{lemma}
\label{extension}
Let $U$ be a probe universal set and let $x \not\in U$ be a
vertex with minimal neighborhood such that
$U^{\prime}=U+N(x)$ is probe
universal with the same number of nonempty label-sets as $U$.
Then there exists an
embedding such that $U$ is universal if and only if there exists
an embedding such that $U^{\prime}$ is universal.
\end{lemma}
\begin{proof}
By definition, the label-sets of $U^{\prime}$ are modules that extend
the label-sets of $U$.
This proves the lemma.
\qed\end{proof}

\begin{theorem}
For each $k$, there exists an $O(n^2)$-time algorithm for the recognition
of $k$-probe threshold graphs.
\end{theorem}
\begin{proof}
We may assume that $G$ has no $k$-probe module.
By Lemma~\ref{super universal} there exists a constant number of
feasible universal sets. By Lemma~\ref{extension}, if there exists a
vertex $x$ that can be labeled such that $N(x)$ extends
the universal set in a way that does not increase the number of nonempty
label-sets
in the universal set,
then the algorithm can greedily extend the universal set with $N(x)$.
Next the algorithm removes the vertex $x$ and tries to find another
greedy extension.

If there are no more greedy extensions,
the algorithm computes the set $\Upsilon$ as in
Lemma~\ref{super universal} in the graph minus
the probe universal set, and
chooses one of the constant number of subsets
as an extension of the probe universal set.
Notice that
there can be at most $2^k$ extensions that increase the number
of label-sets.

Since the computation of maximal vertices can be done in $O(n^2)$ time,
the algorithm can be implemented to run in $O(n^2)$ time.
\qed\end{proof}

\begin{remark}
Perhaps it is a bit surprising that we don't have to
treat the different components of the graph separately.  
\end{remark}

\section{Concluding remarks}
%%%%%%%%%%%%%%%%%%%%%%%%%%%%

The recognition problem of probe interval graphs was introduced by
Zhang {\em et al.\/}~\cite{kn:zhang,kn:mcmorris}. This problem stems
from the physical mapping of chromosomal DNA of humans and other
species. Since then probe graphs of many other graph classes have
been investigated by various authors. In this paper we generalized
the concept to the graph-class-width parameters. So far, we have
limited our research to classes of graphs that have bounded
rankwidth.

In~\cite{kn:hung}, we derived a fixed-parameter algorithm that solves a  
similar problem for the class of 
trivially perfect graphs. It is well-known that every threshold 
graph is trivially perfect. Obviously, this does not imply that the 
algorithm for trivially perfect graphs can be used for threshold graphs.   
In fact, a    
similar, elegant solution as the one that we obtained in this 
paper can not work for threshold graphs. 

For the classes of  
blockgraphs, threshold graphs, trivially perfect graphs, and
cographs we were able to show that the width parameter is
fixed-parameter tractable~\cite{kn:chang,kn:hung,kn:hung1}. One of
the classes for which this is still open is the class of
distance-hereditary graphs. We are unaware of a monadic second-order 
formulation that describes the distance-hereditary width.
Consider a decomposition tree of bounded rankwidth. The `twinset' of a 
branch is defined as 
the subset of vertices that are mapped to the leaves of that branch,  
and that have neighbors in the rest of the graph (outside the branch). 
It is not difficult to 
show that for bounded rankwidth, the graphs that arise as 
twinsets constitute a class of 
graphs that is characterized by a finite collection of forbidden 
induced subgraphs. (For rankwidth one this is the class of 
cographs.) The same holds true for graphs of bounded DH-width. 
So far, we have not been able to describe the class of graphs as 
tree-extensions of these twinsets.  

\section{Acknowledgement}
%%%%%%%%%%%%%%%%%%%%%%%%%

Ton Kloks is currently a guest of the Department of
Computer Science and Information Engineering of National Chung
Cheng University.
He gratefully acknowledges the funding for this research
by the National Science Council
of Taiwan.

\end{document}